\documentclass[10pt,reqno]{amsart}
  \usepackage{geometry}
\usepackage{amssymb,enumerate,color}
\usepackage{color, mathrsfs}
\usepackage{latexsym,amsbsy,bm}

\usepackage{comment}

\usepackage{enumitem}

\input epsf

\usepackage{placeins}
\usepackage{footnote}

\newcommand {\SL} {{\mathbb L}}

\newcommand {\SR} {{\mathbb R}}
\newcommand {\SRd} {{\SR^d}}

\renewcommand {\phi} {{\varphi}}

\newcommand {\al} {{\alpha}}
\newcommand {\dt} {{\delta}}
\newcommand {\Dt} {{\Delta}}
\newcommand {\e} {{\varepsilon}}
\newcommand {\ga} {{\gamma}}
\newcommand {\Ga} {{\Gamma}}
\newcommand {\la} {{\lambda}}

\newcommand{\be}{\mathbf e}

\newcommand{\cA}{\mathcal A}
\newcommand{\cI}{\mathcal I}
\newcommand{\cO}{\mathcal O}
\newcommand{\cD}{\mathcal D}
\newcommand{\cS}{\mathcal S}

\newcommand{\cJ}{\mathcal J}
\newcommand{\cK}{\mathcal K}

\newcommand {\tf} {{\tilde f}}

\newcommand{\fc}{\mathfrak c}
\newcommand{\fa}{\mathfrak a}

\def\Dom{\mathop{\rm Dom}}

\def\sign{\mathop{\rm sign\,}}
\newcommand {\Span} {\mbox{\rm span}\,}

\numberwithin{equation}{section}
\newtheorem{theorem}{Theorem}[section]
\newtheorem{lemma}[theorem]{Lemma}
\newtheorem{defi}[theorem]{Definition}

\newtheorem{Remark}[theorem]{Remark}
\newtheorem{proposition}[theorem]{Proposition}

\newtheorem{example}[theorem]{Example}

\newcommand {\Proofof}[1] {\noindent{\bf P{\footnotesize\bf ROOF} of {#1}: } \ }

\newcommand {\ProofEnd} {
             \begin{flushright} \vskip -0.2in $\Box$ \end{flushright}}

\newcommand{\Ba}[1]{\begin{array}{#1}}
\newcommand{\Ea}{\end{array}}
\newcommand{\Be}{\begin{equation}}
\newcommand{\Ee}{\end{equation}}
\newcommand{\Bea}{\begin{eqnarray}}
\newcommand{\Eea}{\end{eqnarray}}
\newcommand{\Beas}{\begin{eqnarray*}}
\newcommand{\Eeas}{\end{eqnarray*}}
\newcommand{\Benu}{\begin{enumerate}}
\newcommand{\Eenu}{\end{enumerate}}
\newcommand{\Bi}{\begin{itemize}}
\newcommand{\Ei}{\end{itemize}}

\newcommand{\BR}{\begin{Remark} \em}
\newcommand{\ER}{\end{Remark}}
\newcommand{\BE}{\begin{example} \em}
\newcommand{\EE}{\end{example}}

\newcommand {\Ds} {\displaystyle}

\newcommand {\mand} {{\quad\mbox{and}\quad}}

\renewcommand{\th}{{\operatorname{th}\,}}

\newcommand{\sh}{{\operatorname{sh}\,}}

\newcommand{\Da}{\cD^{\al}}
\newcommand{\Das}{\cD_{\rm st}^{\al}}
\newcommand{\Dal}{\cD^{\al}(x_0)}
\newcommand{\Dsig}{\cD^{2\sigma}(x_0)}
\newcommand{\Dsx}{\cD^{2\sigma}_{\rm st}(x)}
\newcommand{\Dx}{\cD^{2\sigma}(x)}
\newcommand{\Dals}{\cD_{\rm st}^{\al}(x_0)}
\newcommand{\Dsigs}{\cD_{\rm st}^{2\sigma}(x_0)}

\newcommand{\Lip}{{\rm Lip}}

\newcommand{\sline}{{\smallskip

\noindent}}

\begin{document}

\title[Pointwise convergence of fractional powers]{Pointwise convergence of fractional powers of Hermite type operators}
\author[Flores, Garrig\'os, Signes,  Viviani]{G. Flores, G. Garrig\'os,  T. Signes,  B. Viviani}

\address{G. Flores, CIEM-CONICET, FaMAF, Universidad Nacional de C\'ordoba, Av. Medina Allende s/n, 
Ciudad Universitaria, CP:X5000HUA C\'ordoba, Argentina}
\email{guilleflores@unc.edu.ar}

\address{G. Garrig\'os and T. Signes, Departamento de Matem\'aticas, Universidad de Murcia, 30100, Espinardo, Murcia, Spain}
\email{gustavo.garrigos@um.es, tmsignes@um.es}

\address{Beatriz Viviani\\
IMAL (UNL-CONICET) y FIQ (Universidad Nacional del Litoral), Colectora Ruta Nac. N 168, Paraje El Pozo - 3000 Santa Fe - Argentina} \email{viviani@santafe-conicet.gov.ar}

\thanks{G.G. was supported in part by grants MTM2017-83262-C2-2-P and
AEI/10.1303/501100011033
from Agencia Estatal de Investigaci\'on (Spain), and grant 20906/PI/18 from Fundaci\'on S\'eneca (Regi\'on de Murcia, Spain).
 }

\date{\today}


\begin{abstract}
When $L$ is the Hermite or the Ornstein-Uhlenbeck operator, we find minimal integrability and 
smoothness conditions on a function $f$ so that the fractional power 
$L^\sigma f(x_0)$ is well-defined at a given point $x_0$. 
We illustrate the optimality of the conditions with various examples. 
Finally, we obtain similar results for the fractional operators $(-\Dt+R)^\sigma$, with $R>0$. 
\end{abstract}

\dedicatory{Dedicated to the memory of Eleonor \emph{Pola} Harboure, whose human qualities \\ and
 professional behavior will always be our guidance}

 \maketitle

\section{Introduction}  \setcounter{equation}{0}

Let $L$ be a positive self-adjoint differential operator densely defined in a Hilbert space $L_2(\Omega,d\mu)$.
Fractional powers $L^\sigma$, for $\sigma>0$, can be defined in various 
(abstract) equivalent ways, 
one of the most standard being via spectral theory:
\Be
\langle L^\sigma f,g\rangle\,=\,\int_{\sigma(L)}\la^\sigma\,dE_{f,g},\label{dE}\quad f\in\Dom(L^\sigma),\;\;g\in L_2(\Omega,d\mu),
\Ee
where $E$ denotes a resolution of the identity associated with $L$; see e.g. \cite[Ch 13]{Rud}. When the spectrum is discrete
$\sigma(L)=\{\la_n\}_{n=0}^\infty$, and $\{\phi_n\}_{n=0}^\infty$ is an orthonormal basis of eigenfunctions, then \eqref{dE} takes the form
\[
L^\sigma f\,=\,\sum_{n=0}^\infty\,\la_n^\sigma\,\langle f,\phi_n\rangle\,\phi_n,
\]
say for $f\in\Span\{\phi_n\}$. Alternatively, it is also possible to express $L^\sigma$ in terms of the contraction 
semigroup $\{e^{-tL}\}_{t>0}$. For instance, when $\sigma\in(0,1)$, via the Bochner formula
\Be
L^\sigma f=\tfrac1{\Ga(-\sigma)}\int_0^\infty (e^{-tL}f-f)\,\frac{dt}{t^{1+\sigma}};\label{Ga}
\Ee 
see e.g. \cite[Ch IX.11]{Yos} and references therein.

When $L$ equals the classical Laplacian $-\Dt$ and $f\in\cS(\SRd)$, both \eqref{dE} and \eqref{Ga} lead to the explicit expression\Be
(-\Dt)^\sigma f(x)\,=\,c_{d,\sigma}\, \mbox{{\small PV}}\int_{\SRd}\frac{f(x)-f(y)}{|x-y|^{d+2\sigma}}\,dy, \quad x\in\SRd,\;\; \sigma\in(0,1),
\label{Dt}\Ee
with a suitable constant $c_{d,\sigma}>0$; see e.g. \cite{Kwa17}. This pointwise formula does actually make sense for a larger class of functions, 
e.g. when $f\in L_1(dy/(1+|y|)^{d+2\sigma})$ and $f$ is H\"older continuous  of order $2\sigma+\e$  near the point $x$.
For general operators $L$, however, explicit expressions such as \eqref{Dt} are not common, and 
the definition of $L^\sigma f$ as in \eqref{dE}-\eqref{Ga} must necessarily be restricted to a suitable 
dense class of functions $f$.

Based on work of Caffarelli and Silvestre \cite{CS}, Stinga and Torrea proposed in \cite{ST} to define $L^\sigma f$ as the 
Neumann boundary value associated with the elliptic PDE
\Be
\left\{\Ba{lll} u_{tt} +\tfrac{1-2\sigma}t\,u_t\,=\,L u, & & t>0
\\
u(0,x)=f(x).&&\\\Ea\right.
\label{pde}\Ee
More precisely, if $\sigma\in(0,1)$ and $\fc_\sigma=2^{2\sigma-1}\Ga(\sigma)/\Ga(1-\sigma)$, they set\Be
L^\sigma f(x)\,=\,-\fc_\sigma\,\lim_{t\to0^+}\,t^{1-2\sigma}\,\partial_tu(t,x),
\label{Lnu}\Ee
 where $u(t,x)$ is the solution of \eqref{pde} given by the Poisson-like integral\Be
u(t,x)=P^\sigma_tf(x):=\frac{(t/2)^{2\sigma}}{\Ga(\sigma)}\,\int_0^\infty e^{-\frac{t^2}{4s}}\,e^{-sL}f(x)\,
\,\frac{ds}{s^{1+\sigma}};\label{pois}
\Ee
see \cite[Theorem 1.1]{ST}. Then the authors specialize to the Hermite operator 
\[
L=-\Dt+|x|^2\quad \mbox{in $L_2(\SRd)$},
\] 
and prove that the \emph{pointwise limit}  in \eqref{Lnu} exists at every $x\in\SRd$, 
and coincides with \eqref{Ga},
whenever $f\in C^2(\SRd)\cap L_1(dy/(1+|y|)^N)$, for some $N>0$; see \cite[Theorem 4.2]{ST}, 
and \cite{ST2,RS16} for slightly less restrictive smoothness assumptions.

\

The purpose of this paper is to show the validity of \eqref{Lnu}  for a wider class of functions $f$, with optimal integrability assumptions and very mild smoothness at a given point $x_0\in\SRd$.
We also consider the slightly more general family of operators
\Be
\label{Hm}
L=-\Dt+|x|^2+m, \quad \mbox{in $L_2(\SR^d)$},
\Ee
with a constant parameter $m\geq -d$ (so that $L$ is positive). For $m=0$ this is the usual Hermite operator, 
while for $m=-d$ it can be transformed (after a change of variables) into the  
Ornstein-Uhlenbeck operator
\[
\cO=-\Dt+2x\cdot\nabla,\quad \mbox{in }\;L_2(\SRd,e^{-|x|^2}\,dx);
\]
see $\S\ref{S_OU}$ below. 

To state our results, we now describe the integrability
and smoothness conditions we shall use below. Given a weight $\Phi(y)\geq0$,  
we say that $f\in L_1(\Phi)$ when 
\Be
\int_{\SR^d}|f(y)|\,\Phi(y)\,dy<\infty.
\label{fphi}
\Ee
Associated with $L=-\Dt+|x|^2+m$ in \eqref{Hm}, and $\sigma\in(0,1)$ we define the weight
\Be
\Phi_\sigma(y)=\left\{\Ba{ll}
\Ds\frac {e^{-|y|^2/2}}{(1+|y|)^{\frac {d+m}2}[\ln(e+|y|)]^{1+\sigma}} & \mbox{ if } m>-d\\
&\\
\Ds\frac {e^{-|y|^2/2}}{[\ln(e+|y|)]^{\sigma}} & \mbox{ if } m=-d.\\
\Ea\right.
\label{phi}
\Ee
Then, $f\in L_1(\Phi_\sigma)$ 
ensures that
the Poisson-like integral $P^\sigma_tf(x)$ in \eqref{pois} is a well-defined (smooth) function when $(t,x)\in(0,\infty)\times\SR^d$,
and this condition is actually best possible for that property;
see \cite[Theorem 1.1]{G17} (or \cite{GHSTV, FV22}).

Next, when $\al\in(0,2)$, we shall say that a (locally integrable) function $f$ is \emph{$\al$-smooth at $x_0\in\SR^d$}, denoted $f\in\cD^\al(x_0)$, if it 
satisfies the Dini-type condition
\[
\int_{|h|\leq\dt} \frac{|f(x_0+h)+f(x_0-h)-2f(x_0)|}{|h|^{d+\al}}\,dh\,<\,\infty,\quad \mbox{for some $\dt>0$}.
\]
We say that $f$ is \emph{strictly $\al$-smooth at $x_0$}, denoted $f\in \cD_{\rm st}^\al(x_0)$, if 
\[
f\in\cD^\al(x_0) \mand
\int_{|h|\leq\dt} \frac{|f(x_0+h)-f(x_0-h)|}{|h|^{d+\al-3}}\,dh\,<\,\infty.
\]
Observe that this last condition is redundant if $d+\al-3\leq0$; 
in particular, if $d=1$ and $\al\in(0,2)$, or $d=2$ and $\al\in(0,1]$.
In other cases however, the classes are different and we shall need this distinction to obtain our results.
We refer to $\S\ref{S_smooth}$ below for several examples of these notions, 
and their relation with local Lipschitz conditions at $x_0$.

Our first main result can now be stated as follows.

\begin{theorem}\label{th1}
Let $L$ be the Hermite operator in \eqref{Hm}, $\sigma\in(0,1)$ and $\Phi_\sigma(y)$ as in \eqref{phi}. Suppose that 
\[
f\in L_1(\Phi_\sigma),\mand f\in\Dsigs\quad\mbox{for some $x_0\in\SR^d$}.
\]
Then, the number $L^\sigma f(x_0)$ exists both, 
in the limiting sense of \eqref{Lnu} and as the absolutely convergent integral in \eqref{Ga}, and both definitions agree.
\end{theorem}

The optimality of the smoothness condition is discussed in \S6 below. The corresponding version for the Ornstein-Uhlenbeck operator  
takes the following form.

\begin{theorem}\label{th2}
Let $L=\cO=-\Dt+2x\cdot\nabla$  and $\sigma\in(0,1)$. Suppose that 
\[
f\in L_1(e^{-|y|^2}/[\log(e+y)]^\sigma),\mand f\in\Dsigs\quad\mbox{for some $x_0\in\SR^d$}.
\]
Then, the number $\cO^\sigma f(x_0)$ exists both, 
in the limiting sense of \eqref{Lnu} and as the absolutely convergent integral in \eqref{Ga}, and both definitions agree.
\end{theorem}

Below, we have attempted to present our results in sufficient generality so that many of the arguments can be applied to
 general operators $L$, while only a few steps require explicit estimates on the kernels.
We illustrate this fact in \S\ref{S_other}, where with a minimal effort we obtain a version of the above theorems for the operator $\SL=-\Dt+R$, with $R>0$;
see Theorem \ref{th3} below.

\section{Preliminary results for general operators $L$}
\label{pL}
In this section we shall assume that $L$ is the infinitesimal generator of a semigroup of operators
$\{e^{-tL}\}_{t>0}$ in $L_2(\SRd)$, and that these are described by the integrals
\Be
e^{-tL}f(x)=\int_{\SR^d}h_t(x,y)f(y)\,dy,
\label{htf}\Ee
for a suitable positive kernel $h_t(x,y)$. 
For each $\sigma\in(0,1)$ we then consider the family of subordinated operators $\{P^\sigma_t=P^{\sigma,L}_t\}_{t>0}$,
formally given by 
\[
P^\sigma_tf:=\tfrac{(t/2)^{2\sigma}}{\Ga(\sigma)}\,\int_0^\infty e^{-\frac{t^2}{4s}}\,e^{-sL}(f)\,
\,\frac{ds}{s^{1+\sigma}}.
\]
More precisely, we let
\Be
P^\sigma_tf(x)=\int_{\SR^d}p^\sigma_t(x,y)f(y)\,dy,
\label{Pnuf}
\Ee
where the corresponding  kernels $p^\sigma_t(x,y)=p^{\sigma,L}_t(x,y)$ are defined by 
\Be
p^\sigma_t(x,y)=\tfrac{(t/2)^{2\sigma}}{\Gamma(\sigma)}\,\int_0^\infty
 e^{-\frac{t^2}{4s}}\,h_s(x,y)\,\frac{ds}{s^{1+\sigma}}.
\label{pt0}\Ee
Observe that a crude estimate such as $0<h_s(x,y)\lesssim s^{-d/2}$ (which will be satisfied by all the operators $L$ we shall use)
guarantees that the integral in \eqref{pt0} is absolutely convergent, and moreover
\Be
\partial_t\big[p^\sigma_t(x,y)\big] \,=\, 
\fa_\sigma \, t^{2\sigma-1}\int_0^\infty \Big(2\sigma-\frac{t^2}{2s}\Big) e^{-\frac{t^2}{4s}}h_s(x,y)\frac{ds}{s^{1+\sigma}},
\label{pt1}
\Ee
with $\fa_\sigma=1/(4^\sigma\Ga(\sigma))$. It is also not hard to show that
\Be
t\,\Big|\partial_t\big[p^\sigma_t(x,y)\big] \Big|\,\lesssim\, p_t^\sigma(x,y)+p_{t/\sqrt2}^\sigma(x,y),
\label{tpt}
\Ee
using the fact that $\sup_{v>0}ve^{-v}<\infty$. However, in order to handle derivatives of the expression $P^\sigma_tf(x)$ in \eqref{Pnuf} we shall need more information 
on the decay of the kernel $p^\sigma_t(x,y)$. 

In the case that $L$ is a Hermite operator the decay is given by the following result from \cite[Lemma 3.1]{G17},
which also clarifies the optimal role of the functions $\Phi_\sigma(y)$ in \eqref{phi}.

\begin{lemma}\label{L1}
Let $L$ be the Hermite operator in \eqref{Hm},  $\sigma\in(0,1)$ and $\Phi_\sigma(y)$ as in \eqref{phi}. 
Then, for every $t>0$ and $x\in\SR^d$, there exist finite numbers $c_1(t,x)>0$ and $c_2(t,x)>0$ such
that \Be c_1(t,x)\, \Phi_\sigma(y) \,\leq\,p^\sigma_t(x,y)\,\leq\,
c_2(t,x)\, \Phi_\sigma(y) \;,\quad \forall\;y\in\SR^d. \label{pk1}\Ee  
\end{lemma}

In this section we wish to keep the general setting for the operator $L$ described above,  but we shall \emph{additionally} assume that, 
for each given $\sigma\in(0,1)$,  the kernel $p^\sigma_t(x,y)$ satisfies 
\Be
p^\sigma_t(x,y)\,\leq\,
c_2(t,x)\, \Phi(y) \;,\quad \forall\;y\in\SR^d. \label{pk2}
\Ee
for some function $\Phi(y)$ and some $c_2(t,x)>0$. 
In all the results in this section we shall not need the explicit expression of $\Phi(y)$.

Our first result will establish a relation between the two \emph{pointwise} definitions of $L^\sigma f(x)$ presented in \eqref{Ga} and \eqref{Lnu} above.
We first consider the following general definition.

\begin{defi}\label{def_Af}{\rm
Let $\sigma\in(0,1)$ and $L$ be an operator such that \eqref{pk2} holds for some  $\Phi(y)$. Given $f\in L_1(\Phi)$, we say that a point $x_0\in\SR^d$ is 
$L^\sigma$-admissible for $f$, denoted  $x_0\in\cA_f(L^\sigma)$, if 
\Be
\int_0^\infty \big|e^{-sL}f(x_0)-f(x_0)\big|\frac{ds}{s^{1+\sigma}}<\infty.
\label{Af}
\Ee
In that case we let 
\Be\label{def_Ls}
L^\sigma f(x_0):=\tfrac1{\Ga(-\sigma)}\int_0^\infty \Big(e^{-sL}f(x_0)-f(x_0)\Big)\,\frac{ds}{s^{1+\sigma}},
\Ee
where $\Ga(-\sigma)=\Ga(1-\sigma)/(-\sigma)$.
}
\end{defi}
\sline

The following result is partially based on the proof of \cite[(4.6)]{ST}.

\begin{proposition}\label{mainPr}
Let $\sigma\in(0,1)$ and $L$ be an operator such that \eqref{pk2} holds for some  $\Phi(y)$.
If $f\in L_1(\Phi)$ and $x\in\cA_f(L^\sigma)$ then 
\Be\label{def_lim}
\lim_{t\to0^+}\,-\fc_\sigma\, t^{1-2\sigma}\,\partial_t\Big[P^\sigma_tf(x)\Big] = L^\sigma f(x)\,,
\Ee
with $\fc_\sigma=2^{2\sigma-1}\Ga(\sigma)/\Ga(1-\sigma)$.
\end{proposition}
\begin{proof}
 Using \eqref{Pnuf}, \eqref{pt1}, \eqref{tpt} and $f\in L_1(\Phi)$ one can justify that 
$$
t^{1-2\sigma}\,\partial_t\Big[P^\sigma_tf(x)\Big] =\fa_\sigma \,\int_{\SR^d}\int_0^\infty \Big(2\sigma-\frac{t^2}{2s}\Big) e^{-\frac{t^2}{4s}}h_s(x,y)\,\frac{ds}{s^{1+\sigma}}\,f(y)\,dy,$$
with $\fa_\sigma=1/(4^\sigma\Ga(\sigma))$. 
We claim that
$$I=\int_0^\infty
\Big(2\sigma-\frac{t^2}{2s}\Big)
e^{-\frac{t^2}{4s}}\frac{ds}{s^{1+\sigma}}=0.$$
 In fact, using the
change $z=t^2/4s$ we see that
$$I=\frac{2\cdot4^{\sigma}}{t^{2\sigma}}\int_0^\infty (\sigma-z)e^{-z}z^{\sigma-1}\,dz=\frac{2\cdot 4^{\sigma}}{t^{2\sigma}}\big(\sigma\Gamma(\sigma)-\Gamma(\sigma+1)\big)=0.$$
Then
\Beas
t^{1-2\sigma}\,\partial_t\Big[P^\sigma_tf(x)\Big] & = & \fa_\sigma\, \int_0^\infty \Big(2\sigma-\frac{t^2}{2s}\Big) e^{-\frac{t^2}{4s}}\,\Big[\int_{\SR^d}h_s(x,y)f(y)\,dy -f(x) \Big]\frac{ds}{s^{1+\sigma}}\\
&=& \fa_\sigma\,\int_0^\infty \Big(2\sigma-\frac{t^2}{2s}\Big) e^{-\frac{t^2}{4s}}\,\Big[e^{-sL}f(x)-f(x)\Big]\frac{ds}{s^{1+\sigma}}.
\Eeas
Since we assume that $x\in\cA_f(L^\sigma)$, by the Lebesgue dominated convergence theorem we can take limits as $t\to0^+$,
and after adjusting the constants one easily obtains the result.
\end{proof}

In view of Proposition \ref{mainPr}, we are interested in finding conditions on a function $f$ which guarantee that a given point $x\in \cA_f(L^\sigma)$. Our next observation shows that only one part of the integral in \eqref{Af} must be checked.

\begin{lemma}\label{L_Ainfty}
Let $\sigma\in(0,1)$ and $L$ be an operator such that \eqref{pk2} holds for some  $\Phi(y)$. Then, for every $A>0$ and every $x\in\SR^d$ there exists  $c(x,A)>0$ such that
\Be\label{hsA}
\int_A^\infty h_s(x,y)\,\frac{ds}{s^{1+\sigma}}\leq c(x,A)\,\Phi(y),\quad y\in\SRd.
\Ee
Moreover, if $f\in L_1(\Phi)$ and $|f(x)|<\infty$ then
\[
\int_A^\infty|e^{-sL}f(x)-f(x)|\,\frac{ds}{s^{1+\sigma}}<\infty. 
\]
\end{lemma}
\begin{proof}
To prove \eqref{hsA}, note that
\[
\int_A^\infty h_s(x,y)\,\frac{ds}{s^{1+\sigma}}\leq e^{\frac1{4A}}\,\int_0^\infty e^{-\frac1{4s}}\,h_s(x,y)\,\frac{ds}{s^{1+\sigma}}= c\, p_1^\sigma(x,y)\leq c(x,A)\,
\Phi(y).
\]
For the last statement,
\Beas
\int_A^\infty\! |e^{-sL}f(x)-f(x)|\frac{ds}{s^{1+\sigma}}& \leq &  \int_{\SRd}|f(y)|\int_A^\infty\! h_s(x,y)\,\frac{ds}{s^{1+\sigma}}\,dy
+|f(x)|\,\int_A^\infty\!\frac{ds}{s^{1+\sigma}}\\
\mbox{{\footnotesize by \eqref{hsA}}} & \lesssim & c(x,A)\int_{\SRd}|f(y)|\,\Phi(y)\,dy +|f(x)|<\infty.
\Eeas
\end{proof}

In order to show that $x_0\in \cA_f(L^\sigma)$ we expect that some smoothness of $f$ at the point $x_0$ must be required. Actually, the smoothness of $f$ will only play a \emph{local} role in the integrals defining the property $\cA_f(L^\sigma)$. This motivates to consider a local notion of $L^\sigma$-admissibility.    

\begin{defi}\label{def_Aloc}{\rm
Let $\sigma\in(0,1)$, and $L$ an operator as above.
Given a locally integrable function $f$, we say that a point $x_0\in\SR^d$ is 
\emph{locally $L^\sigma$-admissible for $f$}, denoted  $x_0\in\cA^{\rm loc}_f(L^\sigma)$, if there exists $\dt>0$ and $A>0$ such that the integrals
\Be\label{Idtf}
\cI_\delta f(x_0,s)=\int_{|x_0-y|<\delta}h_s(x_0,y)\,\big[f(y)-f(x_0)\big]dy,
\quad s\in(0,A)\Ee
satisfy the property
\Be\label{AIdt}
\int_{0}^{A} \big|\cI_\delta f(x_0,s)\big|\,\frac{ds}{s^{1+\sigma}}\,<\,\infty.
\Ee
}
\end{defi}

The next lemma gives decay conditions on the kernel and smoothness of $f$ at $x_0$ that guarantee the validity of the previous property.

\begin{proposition}
\label{Kernel prop.}
Let $\sigma\in(0,1)$, and $L$ an operator as above.
Let $x_0\in\SRd$ be fixed, and assume that the kernel $h_s(x_0,\cdot)$ in \eqref{htf}
satisfies, for some $\dt>0$ and $A\in(0,\infty]$, the estimates
\Be \int_{0}^{A}h_s(x_0,x_0+y)\frac{ds}{s^{1+\sigma}}\,\leq
\frac{c(x_0)}{|y|^{d +2\sigma}},\quad \mbox{when }\;\;|y|\leq\dt\label{C.Tamano}\Ee and
\Be\int_{0}^{A}
\Big|h_s(x_0,x_0+y)-h_s(x_0,x_0-y)\Big|\frac{ds}{s^{1+\sigma}}\,\leq
\,\frac{c(x_0)}{|y|^{(d +2\sigma-3)_+}},\quad \mbox{when }\;\; |y|\leq \dt.
\label{C.Suavidad}\Ee
 Then, for every locally integrable $f$ it holds  
 \Be\label{L_imply}
 f\in\Dsigs \quad \implies \quad x_0\in\cA^{\rm loc}_f(L^\sigma).
 \Ee
 Moreover, \eqref{AIdt} holds with the same $A$ and $\dt$ as in \eqref{C.Tamano} and \eqref{C.Suavidad}.
\end{proposition}

Before proving the result, recall the standard notation
\[
f_{\rm even}(x)=\frac{f(x)+f(-x)}2\mand f_{\rm odd}(x)=\frac{f(x)-f(-x)}2.
\]
We shall use the following elementary lemma.
\begin{lemma}
Let $F$ and $G$ be locally integrable in $\SR^d$, and $B$ a ball centered at the origin. Then
\Be\label{evenodd}
\int_B F(x)G(x)\,dx=\int_B F_{\rm even}(x)G_{\rm even}(x)\,dx +
\int_B F_{\rm odd}(x)G_{\rm odd}(x)\,dx.
\Ee
\end{lemma}
We also introduce the notation
\[
\triangle^1_zf(x)=f(x+z)-f(x-z),\mand \triangle^2_zf(x)=f(x+z)+f(x-z)-2f(x).
\]
Observe that, after dividing by 2, these expressions are respectively
the odd and even parts of the function $z\mapsto f(x+z)-f(x)$.

\

\Proofof{Proposition \ref{Kernel prop.}}
Changing variables $y=x_0+z$ in \eqref{Idtf}, we can write
\[
\cI_\delta f(x_0,s)=\int_{|z|<\delta}h_s(x_0,x_0+z)\,\big[f(x_0+z)-f(x_0)\big]dz.
\]
Then, using the identity in \eqref{evenodd} and simple manipulations, we can rewrite this expression as
\Bea
\cI_\delta f(x_0,s)
& = &  \tfrac12\int_{|z|<\delta} h_s(x_0,x_0+z)\triangle_{z}^{2}f(x_0) dz\nonumber\\
& & \hskip-1cm+\tfrac14\int_{|z|<\delta}\big(h_s(x_0,x_0+z)-h_s(x_0,x_0-z)\big)
\,\triangle_{z}^{1}f(x_0) dz\label{hDt1}
\Eea
Thus, using the kernel assumptions in \eqref{C.Tamano} and \eqref{C.Suavidad}, we clearly
have
\Beas
\int_{0}^{A} |\cI_\delta f(x_0,s)|\,\frac{ds}{s^{1+\sigma}}\,&\leq & \tfrac12\int_{|z|<\delta}|\triangle_{z}^{2}f(x_0)|
\Big(\int_{0}^{A}h_s(x_0,x_0+z) \,\frac{ds}{s^{1+\sigma}}\Big)dz\\
&&\hskip-2cm+\tfrac14 \int_{|z|<\delta}|\triangle_{z}^{1}f(x_0)|\,\int_{0}^{A}|h_s(x_0,x_0+z)-h_s(x_0,x_0-z)|\,\frac{ds}{s^{1+\sigma}}dz\\
&\lesssim &  
\int_{|z|<\delta}\frac{|\triangle_{z}^{2}f(x_0)|}{|z|^{d+2\sigma}}\,dz\,+\,
\int_{|z|<\delta}\frac{|\triangle_{z}^{1}f(x_0)|}{|z|^{(d+2\sigma-3)_+}}\,dz
\Eeas
which is a finite quantity when $f\in\Dsigs$.
\ProofEnd

\begin{Remark}\label{convolution kernel}
{\rm
When the heat kernel at a given $x_0$  satisfies 
\[
h_s (x_0,y)= \rho_{s,x_0} (|x_0-y|),
\]
for some function $\rho_{s,x_0}$ (for instance, if $h_s$ is of convolution type and radial), then the condition \eqref{C.Suavidad} is automatically satisfied (since the integrand is 0). Moreover, in the proof of the proposition the integral in \eqref{hDt1} vanishes, so no bound is needed involving $\triangle_z^1 f(x_0)$. Thus, in that setting, the conclusion \eqref{L_imply} of the proposition holds with the weaker smoothness assumption $f\in \Dsig$. 
}
\end{Remark}

\section{The  Hermite operator $L=-\Dt+|x|^2 + m$, with $m\geq-d$}
\label{S_Her}

In this section we specialize to the case when
\[
L=-\Dt+|x|^2 + m, \quad \mbox{with $m\geq-d$.}
\]
We recall the kernel expressions in this setting. For the heat kernel $h_t(x,y)$, associated with $e^{-tL}$, we have the Mehler formula
\[
h_t(x,y)=\,e^{-tm}\,\frac{e^{-\frac{|x-y|^2}{2\th{2t}}-\th{t} \,x\cdot y}}{[2\pi\sh{2t}]^{\frac d2}}\,,\quad t>0,\;x,y\in\SRd;
\]
see e.g. \cite[(4.3.14)]{Than}. Changing variables to $t=t(s)=\frac12\ln(\frac{1+s}{1-s})$ (or equivalently, $s=\th(t)$), the kernel takes the form
\Be
h_{t(s)}(x,y)=\,\frac{(1-s)^{\frac{m+d}2}}{(1+s)^{\frac{m-d}2}}\, \frac{e^{-\frac14(\frac{|x-y|^2}s + s|x+y|^2)}}{(4\pi s)^{\frac d2}}\,.
\label{mehler2}
\Ee
In the next subsection we shall collect the decay and smoothness estimates of this kernel that will be needed in the proof of Theorem \ref{th1}.

\subsection{Kernel estimates}
\label{S_4.1}

Throughout this section we denote
\Be\label{cK}
\cK(x,y):=\int_0^A h_t(x,y)\,\frac{dt}{t^{1+\sigma}},\quad x,y\in\SRd,
\Ee
where we select $A>0$ so that $\th A=1/2$ (any other $A>0$ would also be fine). 
Performing the change of variables in \eqref{mehler2}
(so that $dt=\frac{ds}
{1-s^2}$) we obtain
\Be
\cK(x,y)=\int_0^{1/2}\,\frac{(1-s)^{\frac{m+d}2-1}\,e^{-\frac14(\frac{|x-y|^2}s+s|x+y|^2)}} {(1+s)^{\frac{m-d}2+1}\,(4\pi s)^{\frac
d2}\,\big(\frac12\ln\frac{1+s}{1-s}\big)^{1+\sigma}}\,ds.
\label{cKs}
\Ee
Observe that in this range of integration we have $\ln\frac{1+s}{1-s}\approx s$, and $1\pm s\approx 1$,
so $\cK(x,y)$ becomes comparable to
\Be
\cK_1(x,y)=\int_0^{1/2}\,\frac{e^{-\frac14(\frac{|x-y|^2}s+s|x+y|^2)}} {s^{1+\sigma+\frac d2}}\,ds.
\label{cKs1}
\Ee
The first lemma shows the decay condition in \eqref{C.Tamano}.
The argument in the proof is similar to the one used in \cite[(4.13)]{GHSTV}
(where a better estimate is obtained).

\begin{lemma}
\label{L_h1}
With the notation in \eqref{cKs}, for every $\sigma>0$ there exists $c=c(\sigma)>0$ such that
\[
\cK(x,y)\leq\frac{c}{|x-y|^{d+2\sigma}}, \quad \forall\;x,y\in\SR^d.
\]
\end{lemma}
\begin{proof}
Changing variables $u=\frac{|x-y|^2}{4s}$ in \eqref{cKs1} we see that
\Beas
\cK(x,y) & \approx & \cK_1(x,y)=\Big(\frac4{|x-y|^2}\Big)^{\sigma+\frac d2}\,\int_{\frac{|x-y|^2}2}^\infty e^{-u}\,e^{-\frac{|x+y|^2\,|x-y|^2}{16u}}\, u^{\sigma+\frac d2}\,\frac{du}u\\
& \leq & \frac{4^{\sigma+\frac d2}}{|x-y|^{d+2\sigma}}\,\int_0^\infty e^{-u}\, u^{\sigma+\frac d2}\,\frac{du}u=
\frac{c}{|x-y|^{d+2\sigma}}.
\Eeas
\end{proof}
\BR\label{R_Wt}
This lemma may also be proved directly from \eqref{cK} using the property 
\[
h_t(x,y)\lesssim t^{-d/2}e^{-c\frac{|x-y|^2}{t}}, \quad 0<t\lesssim1.
\]
This property is known to hold for many other operators $L$.
\ER

We now show the smoothness condition in \eqref{C.Suavidad}.

\begin{lemma}
\label{L_h2}
For every $\sigma>0$, there exists $c=c(\sigma)>0$ such that 
\Be\label{dht}
\int_0^A|h_t(x,x+y)-h_t(x,x-y)|\,\frac{dt}{t^{1+\sigma}}\leq\,\frac{c\,|x|}{|y|^{(d+2\sigma-3)_+}}, \quad \forall\;x\in\SR^d,\;|y|\leq1.
\Ee
\end{lemma}

In the proof of \eqref{dht} we shall use the following elementary inequality.

\begin{lemma}\label{L_ab}
If $x,y\in\SRd$ then
\[
\Big|e^{-|x+y|^2}-e^{-|x-y|^2}\Big|\leq 4|x|\,|y|
\;e^{-\min|x\pm y|^2}.
\]
\end{lemma}
\begin{proof}
Using $|x\pm y|^2=|x|^2+|y|^2\pm2x\cdot y$ we can write
\[
\Big|e^{-|x+y|^2}-e^{-|x-y|^2}\Big|=e^{-|x|^2-|y|^2}\,\Big|e^{2x\cdot y}-e^{-2x\cdot y}\Big|.
\]
Now, letting $t=2|x\cdot y|$, and using the inequality
\[
e^t-e^{-t}=\int^t_{-t}e^s\,ds\leq2te^t, 
\]
we obtain
\[
\Big|e^{-|x+y|^2}-e^{-|x-y|^2}\Big|\leq 4|x|\,|y|\,e^{2|x\cdot y|}\,e^{-|x|^2-|y|^2}= 4|x|\,|y|\,e^{-\min|x\pm y|^2}.
\]
\end{proof}

\Proofof{Lemma \ref{L_h2}}
Denote by $\cK'(x,y)$ the left hand side of \eqref{dht}.
Then, performing the change of variables in \eqref{mehler2}, and disregarding the inessential terms (as discussed before \eqref{cKs1}) we obtain
\Beas
\cK'(x,y) & \approx & \int_0^{1/2}\frac{e^{-\frac{|y|^2}{4s}}\,\big|e^{-\frac s4|2x+y|^2}- e^{-\frac s4|2x-y|^2}\big|}
{s^{1+\sigma+\frac d2}}\,ds \\
\mbox{{\footnotesize (by Lemma \ref{L_ab})}} & \lesssim & 
\int_0^{1/2}\frac{e^{-\frac{|y|^2}{4s}}\,s\,|2x|\,|y|}
{s^{1+\sigma+\frac d2}}\,ds\;\lesssim\; |x|\,|y|\,\int_0^{1/2}\frac{e^{-\frac{|y|^2}{4s}}}
{s^{\sigma+\frac d2}}\,ds\\
\mbox{{\footnotesize ($u=|y|^2/(4s)$)}} & = & 
\frac{c\,|x|\,|y|}{|y|^{d+2\sigma-2}}\,\int_{|y|^2/2}^\infty e^{-u}\,u^{\sigma+\frac d2-1}\,\frac{du}u\;\leq\;
\frac{c'\,|x|}{|y|^{d+2\sigma-3}},
\Eeas
the last bound being valid when $\sigma+d/2-1>0$. This is always the case when $d\geq2$ and $\sigma>0$,
and also if $d=1$ and $\sigma>1/2$. 

In the special case that $d=1$ and $\sigma\in(0,1/2]$, one observes that the integral
\[
I(y):=\int_{|y|^2/2}^\infty e^{-u}\,u^{\sigma-\frac 12}\,\frac{du}u\;\approx \left\{\Ba{lll}
\log(e/|y|), & & \mbox{if $\sigma=1/2$}\\
{|y|^{2\sigma-1}},&& \mbox{if $\sigma\in(0,1/2)$},
\Ea\right.
\] 
when  $|y|\leq1$. Inserting this into the above estimates, it leads to
\[
\cK'(x,y)\lesssim \left\{\Ba{lll}
|x|\,|y|\,\log(e/|y|), & & \mbox{if $\sigma=1/2$}\\
|x|\,|y|,&& \mbox{if $\sigma\in(0,1/2)$},
\Ea\right.
\] 
which implies
\[
 \cK'(x,y)\lesssim|x|, \quad \forall\;|y|\leq1.
\]
Note that this matches \eqref{dht} in the special case $d=1$ and $\sigma\leq 1/2$.
\ProofEnd

Our last result is a strengthening of the decay estimate in Lemma \ref{L_h1} when $|y|\gg|x|$.
The proof follows a similar reasoning as in  \cite[Lemma 4.2]{GHSTV}.

\begin{lemma}    \label{L_h3}
Let $\sigma>0$. Then there exist $c=c(\sigma)>0$ and $\ga>0$ such that
\[
\cK(x,y)\leq c\, e^{-(\frac12+\ga)|y|^2},\quad \mbox{when $|y|\geq 10\,\max\{|x|,1\}$. }
\]
\end{lemma}
\begin{proof}
For simplicity denote $a=|x+y|$ and $b=|x-y|$. Note that, the condition $|y|\geq 10|x|$ implies
\[
a^2, b^2 \geq \big(\tfrac9{10}\big)^2\,|y|^2.
\]
Given a small $\eta\in(0,1)$ (to be determined) we have, for all $s\in(0,1/2)$
\Beas
e^{-\frac14(sa^2+\frac{b^2}s)} & = & e^{-\frac{\eta b^2}{4s}}\,e^{-\frac14(sa^2+(1-\eta)\frac{b^2}s)} \\
& \leq &  e^{-\frac{\eta b^2}{4s}}\,e^{-\frac14\,\big(\frac9{10}\big)^2\,|y|^2\,\big(s+(1-\eta)\frac1s\big)} \\
& \leq &  e^{-\frac{\eta b^2}{4s}}\,e^{-\frac14\,\big(\frac9{10}\big)^2\,|y|^2\,(1-\eta)\,\frac52},
\Eeas
using that $s+\frac1s\geq 5/2$ when $s\in(0,1/2)$. Note that, if $\eta>0$ is chosen sufficiently small,
we can find some $\ga>0$ such that
\[
\tfrac14\,\big(\tfrac9{10}\big)^2\,(1-\eta)\,\tfrac52\,>\,\tfrac12+\ga.
\]
So we have
\[
e^{-\frac14(sa^2+\frac{b^2}s)} \leq 
e^{-\frac{\eta b^2}{4s}}\,e^{-(\frac12+\ga)\,|y|^2},\quad s\in(0,1/2).
\]
Thus, inserting these estimates into  \eqref{cKs1}, we obtain
\[
\cK(x,y)\approx \cK_1(x,y)\leq e^{-(\frac12+\ga)\,|y|^2}\,
\int_0^{1/2}e^{-\frac{\eta |x-y|^2}{4s}}\,\frac{ds}{s^{1+\sigma+\frac d2}}.
\]
Finally, in the last integral we perform the change of variables $u=\frac{\eta |x-y|^2}{4s}$
and obtain 
\[
\cK(x,y)\lesssim  \frac{e^{-(\frac12+\ga)\,|y|^2}}{|x-y|^{2\sigma+d}}\,
\int_0^{\infty}e^{-u}\,u^{\sigma+\frac d2}\,\tfrac{du}u\lesssim e^{-(\frac12+\ga)\,|y|^2},
\]
using in the last step that $|x-y|\approx|y|\geq1$, under the conditions in the statement.
\end{proof}

\subsection{Regular positive eigenvectors}
\label{S_4.2}

\begin{defi}\label{def_eig}{\rm
We say that $\psi(x)\in\Dom(L)$ is a \emph{regular positive eigenvector} of $L$ if
\Benu
\item[(a)] $\psi\in C^\infty(\SRd)$
\item[(b)] $\psi(x)>0$, $\forall\,x\in\SR^d$
\item[(c)] $L(\psi)=\la\psi$, for some $\la\geq0$.
\Eenu
}
\end{defi}
When $L=-\Dt+|x|^2+m$, it is elementary to find an explicit regular positive eigenvector, namely
\[
\psi(x)=e^{-|x|^2/2}.
\]
Indeed, it is easily verified that $L(\psi)=\la\psi$ with $\la=m+d\geq0$.

\

We have the following simple lemma, which is  valid for general operators $L$.

\begin{lemma}
\label{L_psi}
Let $\psi$ be a regular positive eigenvector of $L$. Then, for all $\sigma\in(0,1)$
and all $x\in\SR^d$ it holds
\[
\int_0^\infty\Big|e^{-tL}\psi(x)-\psi(x)\Big|\,\frac{dt}{t^{1+\sigma}}\,<\,\infty.
\]
That is, $x\in\cA_\psi(L^\sigma)$, for all $x\in\SRd$.
\end{lemma}
\begin{proof}
Since $e^{-tL}\psi=e^{-t\la}\psi$,
the result is clear if $\la=0$. If $\la>0$, then 
we have 
\Be\label{aux_psi}
\int_0^\infty\Big|e^{-tL}\psi(x)-\psi(x)\Big|\,\frac{dt}{t^{1+\sigma}}=\psi(x)\,\int_0^\infty|e^{-t\la}-1|\,\frac{dt}{t^{1+\sigma}}.
\Ee
Now, from the elementary estimate \[
|e^{-t\la}-1|=\Big|\int_0^t\,\la\,e^{-s\la}\,ds\Big|\leq \min\{\la\,t, 2\}.
\]
one deduces that \eqref{aux_psi} is a finite expression when $\sigma\in(0,1)$.
\end{proof}

\BR In this paper we shall not pursue this notion with other operators $L$, but it is well-known
that such eigenvectors exist when $L=-\Dt+V(x)$, under very general conditions on $V(x)$; see e.g. \cite[Theorem 11.8]{LiebLoss}.
\ER

\subsection{Proof of Theorem \ref{th1}}
\label{S_4.3}

In this section we prove Theorem \ref{th1} for the Hermite operator
\[
L=-\Dt+|x|^2+m.
\]
That is, if $\sigma\in(0,1)$ and $\Phi_\sigma$ is given as in \eqref{phi}, we must show that, for $x=x_0$,
\[
\int_0^\infty\Big|e^{-tL}f(x)-f(x)\Big|\,\frac{dt}{t^{1+\sigma}}\,<\,\infty,
\]
under the conditions $f\in L_1(\Phi_\sigma)$ and $f\in\Dsx$. In that case, the assertions in the theorem will follow directly from Proposition \ref{mainPr}.
In view of Lemma \ref{L_Ainfty}, it suffices to show that 
\Be
\int_0^A\Big|e^{-tL}f(x)-f(x)\Big|\,\frac{dt}{t^{1+\sigma}}\,<\,\infty,
\label{aux_A}
\Ee
where $A>0$ can be chosen as in \S\ref{S_4.1}.

Let $\psi$ be a regular positive eigenvector for $L$, as described in \S\ref{S_4.2}. Since $\psi(x)>0$, a multiplication by this number does not affect the finiteness of \eqref{aux_A}. Now we have
\Bea
J & := & \psi(x)\,\int_0^A\Big|e^{-tL}f(x)-f(x)\Big|\,\frac{dt}{t^{1+\sigma}}\label{aux_J}\\
& = &  
\int_0^A\Big|\big(e^{-tL}f\big)(x)\psi(x)-f(x)\psi(x)\Big|\,\frac{dt}{t^{1+\sigma}}\nonumber\\
& \leq&   \int_0^A\Big|\big(e^{-tL}f\big)(x)\psi(x)-\big(e^{-tL}\psi\big)(x)f(x)\Big|\,\frac{dt}{t^{1+\sigma}}\nonumber\\
& & \quad + |f(x)|\,\int_0^A\Big|\big(e^{-tL}\psi\big)(x)-\psi(x)\Big|\,\frac{dt}{t^{1+\sigma}}\,=\,J_1+J_2.\nonumber
\Eea
Note that $J_2<\infty$ by Lemma \ref{L_psi}, so we must only prove the finiteness of $J_1$. For that term, we have the following inequalities
\Beas
J_1 & = & \int_0^A\Big|\int_{\SRd}h_t(x,y)\big[f(y)\psi(x)-f(x)\psi(y)\big]\,dy\Big|\,\frac{dt}{t^{1+\sigma}}\\
& \leq & \psi(x)\,\int_0^A\Big|\int_{\SRd}h_t(x,y)\,\big[f(y)-f(x)\big]\,dy\Big|\,\frac{dt}{t^{1+\sigma}}\\
&& \quad +|f(x)|\,\int_0^A\Big|\int_{\SRd}h_t(x,y)\,\big[\psi(x)-\psi(y)\big]\,dy\Big|\,\frac{dt}{t^{1+\sigma}}\,=\,J_{11}+J_{12}.
\Eeas
The two summands, $J_{11}$ and $J_{12}$, can be treated similarly, since both functions $f$ and $\psi$ belong to
$L_1(\Phi_\sigma)\cap\Dsx$, by assumption\footnote{Actually, $\psi$ is much smoother than just $\Dsx$, so $J_{12}$ is formally
easier.}. So in the sequel we will just prove that $J_{11}<\infty$ and this will be enough to conclude the theorem. In fact,
since $\psi(x)>0$, it will suffice to show that
\[
\cJ_{11}:=\int_0^A\Big|\int_{\SRd}h_t(x,y)\,\big[f(y)-f(x)\big]\,dy\Big|\,\frac{dt}{t^{1+\sigma}}\,<\,\infty.
\]
At this point we let $\dt=11\max\{|x|,1\}$, and split the inner integral into the two regions $\{|y-x|<\dt\}$
and $\{|y-x|\geq\dt\}\subset\big\{|y|\geq10\max\{|x|,1\}\big\}$. So recalling the notation for $\cI_\delta f(x,t)$ in \eqref{Idtf} we have
\Be\label{aux_J11s}
\cJ_{11}\leq 
\int_{0}^{A} \big|\cI_\delta f(x,t)\big|\,\frac{dt}{t^{1+\sigma}}\,+\, \cJ_{11}^*,
\Ee
where
\Beas
\cJ_{11}^* & =& \int_0^A\int_{|y|\geq10\max\{|x|,1\}}h_t(x,y)\,\big|f(y)-f(x)\big|\,dy\,\frac{dt}{t^{1+\sigma}}\\
& = & \int_{|y|\geq10\max\{|x|,1\}}\,\big|f(y)-f(x)\big|\,
\cK(x,y)\,dy,
\Eeas
using this time the notation for $\cK(x,y)$ in \eqref{cK}.
By Lemma \ref{L_h3}, this last kernel has a gaussian decay, which leads to 
\[
\cJ_{11}^*\lesssim \int_{\SRd}\big(|f(x)|+|f(y)|\big)\,e^{-(\frac12+\ga)|y|^2}\,dy \lesssim |f(x)| +\int_{\SRd}|f(y)|\,\Phi_\sigma(y)\,dy,
\]
since, from the definition in \eqref{phi}, one has $e^{-(\frac12+\ga)|y|^2}\lesssim \Phi_\sigma(y)$ (actually, for all $\sigma>0$). Thus, the assumption $f\in L_1(\Phi_\sigma)$ gives $\cJ_{11}^*<\infty$, and hence, in view of \eqref{aux_J11s}, we have reduced matters to verify that
\[
\int_{0}^{A} \big|\cI_\delta f(x,t)\big|\,\frac{dt}{t^{1+\sigma}}<\infty.
\]
But this is precisely the condition $x\in\cA_f^{\rm loc}(L^\sigma)$ in Definition \ref{def_Aloc}. Now, in view of 
Proposition \ref{Kernel prop.}, this property is a consequence of the smoothness assumption $f\in\Dsx$, since the heat kernel $h_t(x,y)$
satisfies the two hypotheses in the proposition, \eqref{C.Tamano} and \eqref{C.Suavidad}, due to Lemmas \ref{L_h1} and \ref{L_h2}. This completes the proof of Theorem \ref{th1}.
\ProofEnd

\BR
When $x_0=0$, Theorem \ref{th1} holds with the weaker smoothness condition $f\in\Dsig$.
This is because of the observation in Remark \ref{convolution kernel}, since the heat kernel for the Hermite operator,
 $h_t(0,y)$, only depends on $|y|$; see \eqref{mehler2}. 
\ER

\section{The Ornstein-Uhlenbeck operator  $\cO=\,-\Dt+2x\cdot\nabla $}\label{S_OU}

\subsection{Proof of Theorem \ref{th2}}
We now turn to the proof of Theorem \ref{th2} for the operator
\[
\cO=\,-\Dt+2x\cdot\nabla,  
\]
which is positive and self-adjoint in $L_2(e^{-|y|^2}dy)$.
In this case, there is a well-known transference principle, see e.g. \cite[Prop 3.3]{AbuT06}, that reduces matters
 to the Hermite operator with $m=-d$, that is
\Be L=-\Dt+|x|^2-d,\quad\mbox{in $L_2(\SR^d)$}.\label{LHd}\Ee Indeed, if
we set $\tf(x)=e^{-\frac{|x|^2}2}f(x)$ then it is easily seen that $\cO
f(x)=e^{\frac{|x|^2}2}[L\tf](x)$. Thus,
\[ e^{-t\cO}f(x)\,=\,e^{\frac{|x|^2}2}\,e^{-tL}\tf(x)
\mand P_t^{\sigma,\cO}f(x)\,=\,e^{\frac{|x|^2}2}\,P_t^{\sigma, L}\tf(x), 
\]
so that the convergence properties of
$\cO$ and $L$ in \eqref{LHd} are linked by the mapping
$f\mapsto\tf$. Indeed, just observe that, if
\[
\Phi^{\cO}_\sigma(y)=\frac{e^{-|y|^2}}{[\ln(e+|y|)]^{\sigma}},
\]
then 
\Benu
\item[(i)] $f\in L_1(\Phi^\cO_\sigma)$ iff $\tf\in L_1(\Phi^L_\sigma)$
\item[(ii)] $\cO^\sigma(f)(x)=e^{|x|^2/2}\,L^\sigma(\tf)(x)$, as defined in \eqref{def_Ls}
\item[(iii)] $\Ds \lim_{t\to0^+}t^{1-2\sigma}\,\partial_t\big[P^{\sigma,\cO}_tf(x)\big]=
e^{\frac{|x|^2}2}\,\lim_{t\to0^+}t^{1-2\sigma}\,\partial_t\big[P^{\sigma,L}_t\tf(x)\big].$
\Eenu
Since we also have $f\in\Dsx$ iff $\tf\in\Dsx$, then Theorem \ref{th2} is an immediate consequence of Theorem \ref{th1}.

\ProofEnd

\section{Results for other operators $L$}\label{S_other}

 One can ask whether Theorem \ref{th1} continues to hold for other positive self-adjoint operators $L$. If one aims at \emph{optimal} integrability conditions on $f$, the first step would be to find a suitable function $\Phi^L_\sigma(y)$ such that
 \Be c_1(t,x)\, \Phi^L_\sigma(y) \,\leq\,p^{\sigma,L}_t(x,y)\,\leq\,
c_2(t,x)\, \Phi^L_\sigma(y) \;,\quad \forall\;y\in\SR^d, \label{pkL}\Ee 
 as stated in Lemma \ref{L1}. Such optimal estimates, for certain families of operators $L$,
have already been investigated by the authors (and their collaborators) in earlier papers. For instance, besides the already mentioned reference \cite{GHSTV} for Hermite type operators, we have also considered a large class of \emph{Laguerre type operators} $L$ in \cite{GHSV}, while the Bessel operators (in the case $\sigma=1/2$) were treated by I. Cardoso in \cite{Car}.  

In this paper we have tried to state our results in sufficient generality, so that one part of the arguments can be
applied to general operators $L$ (such as in \S\ref{pL}), while the other parts concern with specific estimates of the 
kernels $h_t(x,y)$ associated with $L$ (such as in \S\ref{S_4.1}). 
We remark that, following this line of reasoning, one can derive versions of Theorem \ref{th1} 
when $L$ is any the aforementioned Laguerre or Bessel operators; we expect to take up these matters in \cite{FGSV_p}.

\

In this section we illustrate this fact in just one specific but particularly simple case. 
For a fixed\footnote{In the sequel we shall not track the dependence of the constants on $R>0$.}  parameter $R>0$, we consider the perturbed Laplacian
\[
\SL=-\Dt+R.
\]
In this case, $e^{-t\SL}$ has a well-known convolution kernel
\Be\label{hR}
h_t(x,y)=e^{-tR}\,W_t(x-y), \quad \mbox{where $W_t(x)=(4\pi t)^{-d/2}\,e^{-\frac{|x|^2}{4t}}$}.
\Ee
It was also proved in \cite[\S5.2]{GHSTV}  that \eqref{pkL} does hold with
\Be
\Phi^{\SL}_\sigma(y):= \frac{e^{-\sqrt{R(1+|y|^2)}}}{(1+|y|)^{ \frac{d+1}2+\sigma}}.
\label{PhiR}
\Ee
We now define a similar kernel as in \eqref{cK} (this time letting $A=\infty$)
\Be\label{cKR}
\cK(x,y):=\int_0^\infty h_t(x,y)\,\frac{dt}{t^{1+\sigma}}.
\Ee
\begin{lemma}\label{L_KR}
With the notation in 
\eqref{hR}, \eqref{PhiR} and \eqref{cKR}, if $x\in\SRd$, there exists $c(x)>0$ such that 
\[
\cK(x,y)\leq\,c(x)\,\Phi^\SL_\sigma(y), \quad \mbox{for all $|y|\geq 2\max\{|x|,1\}$.}
\]
\end{lemma}
\begin{proof}
Note from \eqref{hR} and \eqref{cKR} that
\Be\label{aux_cKR}
\cK(x,y)=(4\pi)^{-d/2}\,\int_0^\infty e^{-tR}e^{-\frac{|x-y|^2}{4t}}\,
\frac{dt}{t^{1+\sigma+\frac d2}}.\Ee
For $\nu>0$, consider the special function
\[
F_\nu(z):=\int_0^\infty e^{-s}e^{-\frac{z^2}{4s}}\,s^{\nu-1}\,ds\,\lesssim (1+z)^{\nu-\frac12}\,e^{-z},\quad  z>0,
\]
where the inequality follows from  the asymptotics of the integral; see e.g. \cite[p. 136]{leb}.
If we change variables $s=|x-y|^2/(4t)$ in \eqref{aux_cKR} we obtain that
\[
\cK(x,y)  = \,c \,\frac{F_{\sigma+ d/2}\big(\sqrt{R}\,|x-y|\big)}{|x-y|^{2\sigma+d}}\,
\lesssim \, \frac{\big(1+\sqrt R\,|x-y|\big)^{\sigma+\frac{d-1}2}}{|x-y|^{2\sigma+d}}\,e^{-\sqrt R\,|x-y|}. 
\]
If we now assume that $|y|\geq 2\max\{|x|,1\}$, the right side is easily seen to be controlled by
$c(x)\,\Phi^\SL_\sigma(y)$; see e.g. \cite[(5.6)]{GHSTV} and subsequent lines for a detailed argument.
\end{proof}

We can now state the corresponding theorem for the operator $\SL=-\Dt+R$.

 \begin{theorem}\label{th3}
 Let $\SL=-\Dt+R$ with $R>0$ fixed. Let $\sigma\in(0,1)$ and $\Phi^\SL_{\sigma}(y)$ be as in \eqref{PhiR}.
 Suppose that
 \[
f\in L_1(\Phi^\SL_{\sigma})\mand f\in\Dsig \quad\mbox{at some $x_0\in\SRd$}.
 \]
 Then $(-\Dt +R)^{\sigma}(x_0)$ is well defined in the limiting sense of \eqref{def_lim}, and as the absolutely convergent integral in \eqref{def_Ls}, and both definitions agree.
 \end{theorem}
 \begin{proof}
 By Proposition \ref{mainPr} we must show that $x=x_0\in\cA_f(\SL^\sigma)$.
Observe that $\psi(y)\equiv1$ is a regular positive eigenvector for $\SL$, according to Definition \ref{def_eig}. 
So, by Lemma \ref{L_psi} and the inequalities following \eqref{aux_J} (applied with $A=\infty$), if suffices to show that
 \Beas
J_1 & = & \int_0^\infty \Big|\big(e^{-t\SL}f\big)(x)-f(x)\big(e^{-t\SL}\psi\big)(x)\Big|\,\frac{dt}{t^{1+\sigma}}\\
& =&   \int_0^\infty\Big|\int_{\SRd}h_t(x,y)\big[f(y)-f(x)\big]\,dy\Big|\,\frac{dt}{t^{1+\sigma}}<\infty.
 \Eeas
We let $\dt=3\max\{|x|,1\}$, and as before, split the inner integral into the regions
$\{|y-x|<\dt\}$ and $\{|y-x|\geq \dt\}
$. So, using the notation for $\cI_\dt f(x,t)$ in \eqref{Idtf} we have
\[
J_1 \,\leq \,  \int_0^\infty\big|\cI_\delta f(x,t)\big|\,\frac{dt}{t^{1+\sigma}}\,+\,J_1^* 
\]
where
\Beas
J_{1}^*  & = &  \int_0^\infty\int_{|y-x|\geq\dt }h_t(x,y)\big|f(y)-f(x)\big|\,dy\,\tfrac{dt}{t^{1+\sigma}}\\
& \leq & 
 \int_{|y|\geq\max\{|x|,1\}
 }\,\big|f(y)-f(x)\big|\,\cK(x,y)\,dy,
\Eeas
with $\cK(x,y)$ as in \eqref{cKR}. By Lemma \ref{L_KR}, 
\[
J^*_1\,\leq\, c(x)\,\int_{\SRd}\big(|f(x)|+|f(y)|\big)\,\Phi^\SL_\sigma(y)\,dy,
\]
which is a finite expression. So we have reduced matters to show that
\[ 
\int_0^\infty\big|\cI_\delta f(x,t)\big|\,\frac{dt}{t^{1+\sigma}}<\infty.
\]
But under the smoothness assumption that $f\in\Dx$, this is a consequence of Proposition \ref{Kernel prop.} (setting $A=\infty$), since the kernel 
$h_t(x,y)$ trivially satisfies \eqref{C.Tamano} (by Remark \ref{R_Wt}) and \eqref{C.Suavidad} (whose left hand side is identically 0; see Remark \ref{convolution kernel}). Finally observe, also by Remark \ref{convolution kernel}, that only the weaker smoothness condition $f\in\Dx$ is used, due to the convolution structure of the kernel $h_t(x,y)$.
 \end{proof}

\section{Smoothness conditions}\label{S_smooth}

In this section we give some examples to illustrate the smoothness conditions
from \S1.
Recall that, for $\al\in(0,2)$, a locally integrable function $f\in\Dal$  if
\[
\int_{|h|\leq\dt} \frac{|f(x_0+h)+f(x_0-h)-2f(x_0)|}{|h|^{d+\al}}\,dh\,<\,\infty,
\]
for some $\dt>0$ (hence for all $\dt>0$).
Also, $f\in\Dals$ if 
\[
f\in\Dal\mand \int_{|h|\leq\dt} \frac{|f(x_0+h)-f(x_0-h)|}{|h|^{d+\al-3}}\,dh\,<\,\infty.
\]
Observe that if $f$ is bounded near $x_0$, this last condition is redundant, so $\al$-smooth and strictly $\al$-smooth agree in this case. 
The two classes also coincide if $d+\al-3\leq 0$, that is 
\Be
\Dal=\Dals, \quad \mbox{if $d=1$, or $d=2$ and $\al\in(0,1]$.}
\label{Deq}
\Ee
In other cases, the classes are different, as shown by the example in \eqref{gx} below. 
Finally, note that strict $\al$-smoothness can also be characterized as follows.

\begin{lemma}
Let $\al\in(0,2)$. Then    $f\in\Dals$ if and only if
\Be\label{Daa}
f\in\Dal\mand \int_{|h|\leq\dt} \frac{|f(x_0+h)-f(x_0)|}{|h|^{d+\al-3}}\,dh\,<\,\infty.
\Ee
\end{lemma}
\begin{proof}
 The implication ``$\Leftarrow$''  is obvious since 
 \[
|f(x_0+h)-f(x_0-h)|\leq |f(x_0+h)-f(x_0)|+|f(x_0-h)-f(x_0)|.
 \]
For the converse implication ``$\Rightarrow$'' note that
\[
2\big(f(x_0+h)-f(x_0)\big) =\Big[f(x_0+h)-f(x_0-h)\Big]+\Big[f(x_0+h)+f(x_0-h)-2f(x_0)\Big].
\]
\end{proof}

We next collect a few further elementary observations.

\

\Benu
\item If $f$ is odd about $x_0$ and $f(x_0)=0$, then $f\in\Dal$, for all $\al\in(0,2)$. For instance, if $\ga\in[0,d)$ then
\Be
\label{fx}
f(x)=\sign(x\cdot\be_1)/|x|^\ga\quad\mbox{if $x\not=0$},\quad f(0)=0,
\Ee
belongs to $\Da(0)$ for all $0<\al<2$, even though it is discontinuous there. However, by \eqref{Daa}, $f\in\Das(0)$ only if 
additionally $\al\in(0,3-\ga)$. In particular, the function 
\Be
\label{gx}
g(x)=\frac{\sign((x-x_0)\cdot \be_1)}{|x-x_0|^{3-\al}}\quad\mbox{if $x\not=x_0$},\quad g(x_0)=0,
\Ee
which is locally integrable if $d+\al-3>0$, shows that $\Dals\subsetneq \Dal$ in the complementary range of \eqref{Deq}. 

\item There exists a function $f\in \Dals$, for all $\al\in(0,2)$, but which is discontinuous and unbounded at $x_0$. Indeed, 
if $d\geq2$ (and $x_0=0$), the function $f$ defined in \eqref{fx} with parameter $\ga=1$ has this property.
If $d=1$, one may take any (locally integrable) odd function unbounded near $x_0$. 

\item If $f\in\Lip_\beta(x_0)$ for some $\beta\in(0,2]$ then $f\in \Dals$ for all $\al<\beta$. Here, $f\in\Lip_\beta(x_0)$, if $\beta\in(0,1]$, means that
\[
|f(x_0+h)
-f(x_0)|\leq c\,|h|^\beta, \quad \forall\,|h|\leq\dt,
\]
for some $c,\dt>0$.
If $\beta\in(1,2]$, it means that $f$ is differentiable at $x_0$ and 
\[
\big|f(x_0+h)
-f(x_0)-\nabla f(x_0)\cdot h\big|\leq c\,|h|^\beta, \quad \forall\,|h|\leq\dt.
\]
Indeed, in either case it is clear that  $f\in\Lip_\beta(x_0)$ implies 
\[
|\triangle_h^2 f (x_0)|=
\big|f(x_0+h)
-2f(x_0)+f(x_0-h)\big|\,\leq\, 2c\,|h|^\beta,
\]
and 
\[
|\triangle_h^1 f (x_0)|=\big|f(x_0+h)-f(x_0-h)\big|\,\leq\, c'\, |h|^{\min\{\beta, 1\}},
\]
which in turn implies  $f\in\Dals$, for all $\al<\beta$. 

\item \label{4} The following examples relate $\Lip_\beta(x_0)$ and $\Dal$ when $\beta=\al$:
\Beas
& & f(x)=|x-x_0|^\al\in\Lip_\al(x_0)\setminus \Dal, \quad \forall\,\al\in(0,2)\\
& & g(x)=\sign\big((x-x_0)\cdot\be_1\big)\,|x-x_0|^\al\in \Lip_\al(x_0)\cap\Dals, \quad \forall\;\al\in(0,2).
\Eeas

%
%

\item We mention an example relating the above smoothness conditions at a point $x_0$ and the existence 
of $(-\Dt)^\frac\al2f(x_0)$, as defined in \eqref{Dt}. Consider the two functions $f,g$, defined as in point \eqref{4} above but 
additionally multiplied by a smooth cut-off $\phi(|x-x_0|)$, where $\phi\in C^\infty_c(\SR)$ with $\phi\equiv 1$ in $[-1,1]$. 
Then, it is easily seen that 
\[
(-\Dt)^\frac\al2f(x_0)=-\infty \quad \mbox{but} \quad (-\Dt)^\frac\al2g(x_0)=0.
\] 
So in general, $f\in\Lip_{\al}(x_0)$ is not enough to define pointwise fractional powers, $L^\frac{\al}2f(x_0)$, 
justifying the search for a different condition such as $f\in\Dal$ or $f\in\Dals$.

\item When $L$ is the Hermite operator in \eqref{Hm}, one can show that the function $g(x)$ defined in \eqref{gx},
when  additionally multiplied by a smooth cut-off $\phi(|x-x_0|)$ supported near the point $x_0=\be_1$, has the property
\[
L^\frac\al2 g(x_0)=\infty.
\] 
Thus, when $x_0\not=0$, in Theorem \ref{th1} one cannot replace the assumption $f\in\Dsigs$ by the weaker condition $f\in \Dsig$
(except of course when the two classes coincide; see \eqref{Deq}). This example also shows that there are compactly supported $g$ such that
\[
(-\Dt)^\frac\al2 g(x_0)=0\quad \mbox{but}\quad L^\frac\al2 g(x_0)=\infty
\]
(the latter in the sense of Definition \ref{def_Af}).
\Eenu

\subsection*{Acknowledgements}
We wish to thank the referee for useful comments which have led to an improved version of this paper.

 \end{document}